\documentclass[12pt]{article}

\RequirePackage{amsthm,amsmath,amsfonts,amssymb}
\RequirePackage[numbers]{natbib}
\RequirePackage{enumerate,multirow,bbm,bm,subfig,graphicx}
\RequirePackage[colorlinks]{hyperref}

\theoremstyle{plain}
\newtheorem{theorem}{Theorem}[section]
\newtheorem{corollary}[theorem]{Corollary}

\newtheorem{thm}{Theorem}
\newtheorem{lem}[thm]{Lemma}
\newtheorem{definition}{Definition}

\theoremstyle{remark}

\newtheorem*{remark}{Remark}

\renewcommand{\hat}{\widehat}
\renewcommand{\tilde}{\widetilde}

\DeclareMathOperator*{\argmin}{argmin}

\usepackage{mathtools}
 \newcommand{\esp}[1]{\mathbb{E}\left[#1\right]} 
 
\newcommand{\prob}[1]{\mathbb{P}\left(#1\right)} 
 
\newcommand{\R}{\mathbb{R}} 

\begin{document}

\begin{center}
	\Large \bf  GROS: A General Robust Aggregation Strategy.
\end{center}

 \normalsize
\begin{center}
	Alejandro Cholaquidis$^1$,    Emilien Joly$^2$ and Leonardo Moreno$^3$\\
	$^1$ Centro de Matemática, Facultad de Ciencias, Universidad de la República, Uruguay\\
	$^2$ Centro de Investigación en Matemáticas, CIMAT, México.\\
	$^3$ Departamento de M\'{e}todos Cuantitativos, Facultad de Ciencias Econ\'{o}micas y de Administraci\'{o}n, Universidad de la Rep\'{u}blica, Uruguay.
\end{center}

\begin{abstract}
A new, very general, robust procedure for combining estimators in metric spaces is introduced GROS. The method is reminiscent of the well-known median of means, as described in \cite{devroye2016sub}.  Initially, the sample is divided into $K$
groups. Subsequently, an estimator is computed for each group. Finally, these $K$ estimators are combined using a robust procedure. We prove that this estimator is sub-Gaussian and we get its break-down point, in the sense of Donoho. The robust procedure involves a minimization problem on a general metric space, but we show that the same (up to a constant) sub-Gaussianity is obtained if the minimization is taken over the sample,  making GROS feasible in practice.
The performance of GROS is evaluated through five simulation studies: the first one focuses on classification using $k$-means,  the second one on the multi-armed bandit problem, the third one on the regression problem. The fourth one is the set estimation problem under a noisy model. Lastly, we apply GROS  to get a robust persistent diagram.
\end{abstract}

\noindent%

\section{Introduction}

The problem of combining estimators has been extensively studied in statistics. There are recent proposals that merge regression estimators (see, for instance,  \cite{biau2016cobra}), classifiers (see \cite{cholaquidis2016nonlinear}), and density estimators (see  
 \cite{cholaquidis2021combined}), among others. In these scenarios, the aim is to merge the estimators to generate one that, at least asymptotically, surpasses the best of the group. In other instances, the  aim  is to derive a robust estimator.  This is the case of the well-known \textit{median of means} (MOM) estimator, which aims to derive a robust estimator of the mean of a random variable.  In the MOM the data, $\aleph_n$ (an i.i.d. sample of a random variable $X$), is first randomly partitioned into $K$ groups. Subsequently, the mean of each group is computed. The MOM estimator is then the median of these  $K$ means. If the variance of the data is assumed to be finite,  this estimator is sub-Gaussian. For further details, we refer to \cite{devroye2016sub}, \cite{joly2017estimation}, and the references therein. For the case of random vectors, the so-called median of means tournament is introduced in \cite{lugosi2019sub}, where is proved to be sub-Gaussian.  In \cite{rodriguez2019breakdown} it is proved that the median of means tournament has break-down point $\lfloor (K-1)/2\rfloor /n$, where $\lfloor x\rfloor$ denotes the floor of $x$ and $n$ is the sample size.

 The use of these estimators has proven to be valuable in varied statistical scenarios, such as in machine learning, see \cite{lecue2020}.  In these contexts, it is advisable to consider estimators that, without removing outliers, do not reduce their precision. Robust statistics point in this direction, see \cite{maronna2019} and \cite{aaron2019}.

Following this idea of dividing into $K$ groups, calculating the estimator in each group, and then combining them, we will introduce a new way to combine the estimators in order to obtain, in a very general framework, a new one that, under the assumption that the estimators by group are independent, turns out to be sub-Gaussian (see Theorem \ref{thm:aggregation_theoretical} below). This new strategy,  in what follows: GROS,  has breakdown point $ \lceil K/2\rceil /n$, where $\lceil x\rceil$ denotes the ceiling of $x$, see Section \ref{sec:robagg}. The only assumption we make is that the original sample comes from a random  variable with finite variance and that the space where the group estimators take their values is a separable and complete metric space.

While the combined estimator requires solving a minimization problem in a metric space, we prove that if it is minimized on the sample of the group estimators, an appropriate candidate is obtained. We also determine how much is lost by this choice, see Section \ref{sec:compasp}.

Due to the immense generality of GROS , it can be applied to various areas of statistics where robustness plays a key role. Furthermore, the space in which the estimators reside doesn't need to be a metric space: a pseudometric suffices. This permits the consideration of estimators in the space of bounded subsets of $\mathbb{R}^d$  equipped with the Hausdorff distance or the measure distance, which is the case when the object to be estimated is, for instance, a set. This space will be used in our fourth simulation example, see subsection \ref{sec:setesti}.

We have chosen to present five problems to demonstrate its performance, comparing it with techniques explicitly crafted for these specific issues. Some of these techniques were already designed to yield robust estimators. Specifically, we treat:
\begin{itemize}
    \item The traditional clustering problem;
    \item The multi-armed bandit problem with heavy-tailed rewards;
    \item Regression in the presence of noisy data;
    \item The estimation of a convex set when dealing with a noisy sample;
    \item An application to topological data analysis.
\end{itemize}
    It is worth noting that the fourth problem cannot be successfully treated using conventional methods such as convex hulls or $r$-convex hulls, as we will see.

As expected, in all cases the performance of GROS  is noticeably better than the proposals that do not consider the presence of outliers, see Section \ref{sec:applications}. Moreover, the performance is good compared to methods that do consider the presence of outliers, even surpassing some in certain cases.

	The structure of the paper is outlined as follows: Section 2 introduces GROS  within a broad context and examines its robustness properties. Section 3 presents a modification of GROS  to simplify its computational aspects. In Section 4, the application of GROS  across five problems is discussed. The paper concludes with Section 5, where the findings and implications of the study are elaborated.

\section{Robust aggregation of weakly convergent estimators} \label{sec:robagg}

In this section, we define and study a new proposal of a robust estimator based on the aggregation of estimators. We assume given a sample $\aleph_n=\{X_1,\dots,X_n\}$ of i.i.d. random elements with common distribution $P$.  Let $\mu$ be a certain characteristic of  $P$. We assume that $\mu$ belongs to a  complete and separable metric space $\mathcal{M}$ endowed with a metric $d$. As we mentioned in the introduction, all the results in this paper remains true if $d$ is a pseudo metric.
In the context of robust estimation, one goal is to obtain sub-Gaussian type  inequalities for the deviation of an estimator $\hat{\mu}$ from $\mu$. A common way of defining a robust, distribution-free estimator is to make $K$ disjoint groups out of $\aleph_n$, hence, to create a collection of $K$ independent estimators $\mu_1,\dots,\mu_K \in \mathcal{M}$. We define the robust aggregation of $\mu_1,\dots,\mu_K$, GROS, by
\begin{equation}
    \label{eq:def}
    \mu^*= \argmin_{\nu\in \mathcal{M}}\min_{I:|I|> \frac{K}{2}} \max_{j\in I} d(\mu_j,\nu).
\end{equation}
The minimization is taken over all the possible subsets $I$  of $\{1,\dots,K\}$ that contain at least $\lfloor  K/2 \rfloor+1$ indices. However, it is easy to see that it is enough to minimize over all possible subsets $I$ whose cardinality, $|I|$, fulfills $|I|=\lfloor  K/2 \rfloor+1$.
Observe that, for any $\nu\in \mathcal{M}$,
\[
\min_{I:|I|> \frac{K}{2}} \max_{j\in I} d(\mu_j,\nu)=:d(\nu,\nu_{(\lfloor  K/2 \rfloor+1)-\text{NN}}),
\]
where $\nu_{(\lfloor  K/2 \rfloor+1)-\text{NN}}$ denotes the  $(\lfloor  K/2 \rfloor+1)$-nearest neighbor of $\nu$ in $\mu_1,\dots,\mu_K$. This last quantity is a measure of the depth of $\nu$ inside the set  $\mu_1,\dots,\mu_K$. Then, $\mu^*$ is the point with the least depth from all the candidates $\nu\in \mathcal{M}$. 
In full generality, the set of minimizers of \eqref{eq:def} may not be unique. In that case, we still denote by $\mu^*$ one of the minimizers arbitrarily chosen.  Since we assumed that $\mathcal{M}$ is a Polish space, it is easy to see that $\mu^*\neq \emptyset$.

\begin{remark}
A natural generalization of $\mu^*$ is to define, for $q\in [1/2,1]$,
\begin{equation*}
    \mu_q^*= \argmin_{\nu\in \mathcal{M}}\min_{I:|I|> Kq} \max_{j\in I} d(\mu_j,\nu).
\end{equation*}
All the results we present are for $q=1/2$ but they remain true for any $q\in [1/2,1]$.
\end{remark}

As we said in the Introduction, we aim to combine the estimators $\mu_1,\dots,\mu_k$ in a robust way. More precisely, let us recall the definition of finite-sample breakdown point introduced by Donoho (see \cite{donoho1982breakdown}).
\begin{definition} Let $\mathbf{x}=\{x_1,\dots,x_n\}$ be a dataset, $\theta$ an unknown parameter lying in a metric space $\Theta$, and $\hat{\theta}_n=\hat{\theta}_n(\mathbf{x})$ an estimator based on $\mathbf{x}$.
Let $\mathcal{X}_p$ be the set of all datasets $\mathbf{y}$ of size $n$ having $n-p$ elements in common with $\mathbf{x}$:
	$$\mathcal{X}_p=\{\mathbf{y}:|\mathbf{y}|=n  \text{ and } |\mathbf{x}\cap \mathbf{y}|=n-p\}.$$
	Then, the breakdown point of $\hat{\theta}_n$ at $\mathbf{x}$ is $\epsilon_n^*(\hat{\theta}_n,\mathbf{x})=p^*/n,$
	where 
\begin{multline*}
p^*=\max\{p\geq 0: \forall \mathbf{y}\in \mathcal{X}_p, \hat{\theta}_n(\mathbf{y}) \textrm{ is bounded and} \\  \text{bounded away from the boundary} \ \partial \Theta,   \textrm{ if } \partial \Theta\neq \emptyset  \}.
\end{multline*}
\end{definition}

From  \eqref{eq:def} it follows easily that the  finite-sample breakdown point of $\mu^*$ is $\lceil K/2\rceil/n$, which is the same order obtained in \cite{rodriguez2019breakdown} for the MOM aggregation strategy mentioned in the Introduction.

The following lemma states that if there exists  an $\eta$ for which at least $K/2$ of the $\mu_i$ are at a distance at most $t$ from $\eta$, then any minimum in \eqref{eq:def} is at a distance at most $2t$ from $\eta$. This, as we will see, implies the  robustness and sub-Gaussianity of the estimator \eqref{eq:def}.
Let us write $[K]=\{1,\dots,K\}$.

\begin{lem} \label{lem:binomial_argument} Assume that there exist  an $\eta \in \mathcal{M}$ and an $I\subset [K]$ with $|I|>K/2$ such that for all $j\in I$, $d(\mu_j,\eta)\le t$. Then, $d(\mu^*,\eta)\le 2t$.
\end{lem}
\begin{proof} By hypothesis, there exists a set $I$ of cardinality at least $K/2$ such that $\max_{j\in I} d(\mu_j,\eta) \le t$. Since $\mu^*$ is a minimizer of \eqref{eq:def}, there exists a set $I_0$ (a priori different from $I$) with $|I_0|>K/2$ such that $\max_{j\in I_0} d(\mu_j,\mu^*)\le t.$
Now, note that $|I|+|I_0|>K$ and so there exists $j\in I\cap I_0$ such that  $d(\mu^*,\eta)\le d(\mu^*,\mu_j)+d(\mu_j,\eta)\le 2t$, which concludes the proof.
\end{proof}

Lemma \ref{lem:binomial_argument} can be applied when $\eta=\mu$, in which case if a group of at least $K/2$ estimators is reasonably close to the objective $\mu$, then $\mu^*$ itself is reasonably close.
Such an estimator is robust to outliers since this effect will not be altered by the bad behavior of up to $K/2-1$ estimators. This lemma is a technical fact that will allow us to use the so called binomial argument. Indeed, assume that $\mu$ is such that for any $0< p < 1/2$, there exists $t=t(n,K)$ such that for all $k\in 1,\dots,K$, $\prob{d(\mu_k,\mu)> t} \le p$.
Since the estimators $\mu_1,\dots,\mu_K$ are independent,
 
\begin{align}\label{boundbinom}
    \prob{d(\mu^*,\mu)> 2 t} &\leq  \mathbb{P} (\exists I: |I|\ge \lfloor K/2 \rfloor  \ \text{ and } \forall i\in I,d(\mu_i,\mu)>t) \nonumber\\
    &=\mathbb{P}\Bigg(\sum_{k=1}^K \mathbb{I}_{\{d(\mu_k,\mu)>t\}} \geq \lfloor K/2\rfloor \Bigg) \nonumber\\
    &\le \prob{B_{K,p}\ge \lfloor K/2\rfloor } \nonumber \\
    &\le e^{\frac{-2(\lfloor K/2\rfloor-Kp)^2}{K}},
\end{align}
where $B_{K,p}$ denotes a random variable with binomial distribution, with parameters $K$ and $p$. Let us assume that the $\mu_i$ are identically distributed. Then, if we choose $p=1/4$ and $K=\lceil 8\log(\delta^{-1})\rceil$, using the fact that $\lfloor \lceil x\rceil/2\rfloor\leq \lceil x\rceil /2$, we get from \eqref{boundbinom} together with Markov's inequality  
$$\prob{d(\mu^*,\mu)> 4\sqrt{\mathbb{E} d^2(\mu_1,\mu)}}\leq \delta.$$ 
Lastly, we have proved the following theorem:
\begin{thm}
\label{thm:aggregation_theoretical}
Let $\aleph_n=\{X_1,\dots,X_n\}$ be an i.i.d. sample from a distribution $P$.  Let $\mu\in \mathcal{M}$ a certain characteristic of $P$ where $(\mathcal{M},d)$ is a metric space. Let $\delta>0$ and $K=\lceil 8\log(\delta^{-1})\rceil$. We split $\aleph_n$ into $K$ disjoint groups (assume that $n$ is large enough to guarantee that $n/K=\ell \in \mathbb{N}$), and create  $K$ estimators  $\mu_1,\dots,\mu_K$ of $\mu$. Let $\mu^*$ be the aggregation defined by \eqref{eq:def}. Then,  
   \begin{equation}\label{concentration}
 \prob{d(\mu^*,\mu)> 4\sqrt{\mathbb{E} d^2(\mu_1,\mu)}}\leq \delta.
   \end{equation} 
\end{thm}

\begin{remark} In Theorem \ref{thm:aggregation_theoretical}, the restriction that $n$ is large enough to guarantee $n/K=\ell \in \mathbb{N}$ is purely technical, to ensure that the $\mu_1,\dots,\mu_K$ are identically distributed and then to get the clean expression  \eqref{concentration}. If this is not the case, that is, the groups are unbalanced, the estimators $\mu_1,\dots,\mu_K$ are independent but not necessarily identically distributed, and the obtained bound is 
   \begin{equation}\label{concentration2}
 \prob{d(\mu^*,\mu)> 4\sqrt{ \max_{i=1,\dots,K}\mathbb{E} d^2(\mu_i,\mu)}}\leq \delta.
   \end{equation} 
\end{remark}

\subsection{Mis-specification of the set $\mathcal{M}$} \label{missesp}
 
Lemma \ref{lem:binomial_argument} uses the fact that $\eta$ is inside the set $\mathcal{M}$. Now, assume given a imperfect set $\tilde{\mathcal{M}}$ that is mis-specified in the sense that $d(\eta,\tilde{\mathcal{M}})=\inf_{\nu\in \tilde{\mathcal{M}}} d(\eta,\nu)=\epsilon>0$. The true parameter of interest does not belong to the set of features $\tilde{\mathcal{M}}$. The minimization is then given by
\begin{equation}
    \tilde{\mu}= \argmin_{\nu\in \tilde{\mathcal{M}}}\min_{I:|I|> \frac{K}{2}} \max_{j\in I} d(\mu_j,\nu).
\end{equation}

\begin{lem} \label{lem:missesp} Assume that there exist an $\eta \in \mathcal{M}$ and an $I\subset [K]$ with $|I|>K/2$ such that for all $j\in I$, $d(\mu_j,\eta)\le t$. Then, $d(\tilde{\mu},\eta)<2t+\epsilon$.
\end{lem}
\begin{proof}
By the triangle inequality $\forall \delta>0$, there exists  an $\eta_\delta \in \tilde{\mathcal{M}}$ such that $$\max_{j\in I} d(\mu_j,\eta_\delta) \le t+\epsilon+\delta.$$ So there exists a set $I_0$ of cardinality $|I_0|>K/2$, such that $
\max_{j\in I_0} d(\mu_j,\tilde{\mu}) \le t+\epsilon+\delta.$

Since $|I|+|I_0|>K$, there exists $j_0\in I\cap I_0$. Then, $d(\tilde{\mu},\eta)\leq d(\tilde{\mu},\mu_{j_0})+d(\mu_{j_0},\eta)\leq 2t+\epsilon+\delta$.
Since this holds for all $\delta>0$, it follows that $d(\tilde{\mu},\eta)<2t+\epsilon$.
\end{proof}
The following corollary is a direct consequence of Lemma \ref{lem:missesp}.
\begin{corollary}
Let $\aleph_n=\{X_1,\dots,X_n\}$ be an i.i.d. sample from a distribution $P$.  Let $\mu\in \mathcal{M}$ be a certain characteristic of $P$, where $(\mathcal{M},d)$ is a metric space. We assume given a set $\tilde{\mathcal{M}}\subset \mathcal{M}$ (possibly random) such that $d(\eta,\tilde{\mathcal{M}})=\epsilon>0$. Let $\delta>0$ and $K=\lceil 8\log(\delta^{-1})\rceil$. We construct the $K$ disjoint groups and $K$ estimators  $\mu_1,\dots,\mu_K$ of $\mu$ as in Theorem \ref{thm:aggregation_theoretical}. Then,  
   \begin{equation}
 \prob{d(\tilde{\mu},\eta)> 4\sqrt{\mathbb{E} d^2(\mu_1,\mu)}+\epsilon}\leq \delta.
   \end{equation} 
\end{corollary}

\section{Computational aspects}\label{sec:compasp}

Equation \eqref{eq:def} supposes that one is able to find  minimizers of a complex functional on the metric space $\mathcal{M}$, which is often an unfeasible problem. To simplify that task, one can restrict  the minimization to the set of estimators  $\mu_1,\dots,\mu_K$. That is, we find the index $j^*$ such that
\begin{equation}
\label{eq:def_jstar}
    j^*=\argmin_{j=1,\dots,K}\min_{I:|I|> K/2} \max_{i\in I} d(\mu_i,\mu_j).
\end{equation}
The next lemma and theorem state that  $\mu_{j^*}$  has the same sub-Gaussian type bound (up to a constant) as $\mu^*$.

\begin{lem}
\label{lem:binomial_practical}
    Assume that there exists an $I\subset [K]$ such that $|I|> K/2$, and for all $j \in I$, $d(\mu,\mu_j) \le t$. Then  $d(\mu_{j^*},\mu)\le 3t$.
\end{lem}

\begin{proof}
Let $I \subset [K]$ be such that $|I|> K/2$, and for all $j  \in I$, $d(\mu,\mu_j) \le t$, we have $d(\mu_i,\mu_j)\leq d(\mu_i,\mu)+d(\mu,\mu_j)\leq 2t$ for all $i,j\in I$. Then, there exists an $I_0$ with cardinality at least $K/2$ such that $d(\mu_{j^*},\mu_i)\leq 2t$ for all $i\in I_0$. Since $|I|+|I_0|>K$, there exists $j_0\in I\cap I_0$. Lastly, $d(\mu_{j^*},\mu)\leq d(\mu_{j^*},\mu_{j_0})+d(\mu_{j_0},\mu)\leq 3t$.
\end{proof}

By means of Lemma \ref{lem:binomial_practical}, it is possible to give a practical version of Theorem \ref{thm:aggregation_theoretical}.
\begin{thm} Assume the  hypotheses of Theorem \ref{thm:aggregation_theoretical}. Let $\mu_{j^*}\in\{\mu_1,\dots,\mu_K\}$ defined by the optimization \eqref{eq:def_jstar}. Then,
    \begin{equation}
    \label{eq:practical_concentration}
        \prob{d(\mu_{j^*},\mu)> 6\sqrt{\mathbb{E} d^2(\mu_1,\mu)}}\leq \delta.
    \end{equation}
\end{thm}
Note that Equation \eqref{eq:practical_concentration} is the same as Equation \eqref{concentration} except that it has the constant 6 in the right-hand side. This shows that the practical version of the estimator, $\mu_{j^*}$, has essentially the same rate of convergence as $\mu^*$, but it can be deteriorated by a constant factor.

\section{Some applications of GROS}\label{sec:applications}

\subsection{Classification by $k$-means}

One of the most popular procedures for determining clusters in a dataset is the $k$-means method. Although the precursors of this algorithm were MacQueen in 1967, see \cite{mcqueen67}, and Hartigan in 1978, see \cite{hartigan78}. Pollard in 1981 \cite{pollard81} proved the strong consistency of the method and in 1982, in \cite{pollard82}, determined its asymptotic distribution.

Given $k$, the $k$-means clustering procedure partitions a set $\{x_1, \ldots, x_n\}\subset \mathbb{R}^d$ into $k$ groups as follows:   first $k$ cluster centres $a_j$ are chosen, in such a way as to minimise

  $$W_n = \frac{1}{n} \sum_{l = 1}^n \min_{1 \leq j \leq k} \Vert x_l - a_j \Vert^2.$$
Then it assigns each $x_l$ to its nearest cluster centre. In this way, each centre  acquires a subset $C_l$  as its associated cluster. The mean of the points in $C_l$ must equal $a_l$, otherwise $W_n$ could be decreased by replacing $a_l$ with the cluster mean in the first instance, and then reassigning some of the $x$'s to their new centres. This criterion is then equivalent to that of minimising the sum of squares between clusters.

The standard $k$-means algorithm starts from a set   $a_1^{(1)}, \ldots , a_k^{(1)},$ and alternates, up to a stopping criterion, two steps: \textit{Assignment step:} Assigns each observation $x_j$ to the cluster whose centre is the closest one. \textit{Update step:} Recalculate means (centroids) for the observations assigned to each cluster, by averaging the observations in each cluster.

These $k$-centres are not robust to the presence of outliers, nor to the distribution's  possession of heavy tails.
There are several proposals in the literature that seek to make the $k$-means algorithm more robust. An example is the $k$-medoids algorithm (called PAM), see \cite{kaufman1990,kaufman2009}. In this algorithm, in the second step, one chooses that point in the cluster which minimises the sum of the distances to the remaining observations. These points are called the medoids.

Another proposal for a robust version of $k$-means, TClust,  is developed in \cite{cuesta1997}. It is based on an $\alpha>0$ trimming of the data.  This trimming is self-determined by the data  and aims to mitigate the impact of extreme data.

We propose a simple modification to the second step of the $k$-means algorithm, which will be referred to as \textit{RobustkM}. The idea is very simple: instead of calculating the centroids by taking the arithmetic mean in each group, the centroids are determined using \eqref{eq:def_jstar}.

\subsubsection{Simulations}

To evaluate the performance of this proposal, we run a small simulation study. In all cases, the data consist of an i.i.d. sample $X_1, \ldots, X_n$, whose common distribution, $F_X$, is given by the mixture of three bi-variate Student distributions. More precisely,

\begin{equation}\label{student}F_X(x)=  0.45 T \left( x,\mu_1, \nu, \Sigma_1 \right) + 0.45 T \left( x,\mu_2, \nu, \Sigma_2 \right) +0.1 T \left( x,\mu_3, \nu, \Sigma_3 \right), 
\end{equation}

\noindent where $T \left( x, \mu, \nu, \Sigma \right)$ denotes the bi-variate cumulative Student distribution function with mean $\mu \in \mathbb{R}^2$, variance and covariance matrix $\Sigma$, and $\nu=2$ degrees of freedom.

In the examples, we chose $\mu_1=(6,0)$,  $\mu_2=(-6,0)$, $\mu_3=(0,6)$, $\Sigma_1=\Sigma_2= \bigl( \begin{smallmatrix}3 & 0\\ 0 & 3\end{smallmatrix}\bigl)$ and $\Sigma_3= \bigl( \begin{smallmatrix} 4 & 1\\ 1 & 9\end{smallmatrix}\bigl)$.

Figure \ref{fig_mediasp} shows a simulation with $n=1000$ points. It can be seen that the dispersion of the third group makes the clustering problem more difficult.

\begin{figure}[htb]
\centering
 \subfloat{\includegraphics[width=110mm]{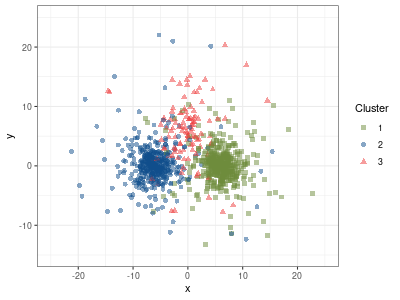}}
\caption{Simulation of 1000 observations of the multivariate Student mixture \eqref{student}. Observations are colored according to the component of the mixture which the data comes from.}
\label{fig_mediasp}
\end{figure}

Let $\pi$ be a permutation of the set of labels $\{1,2,3\}$. Denote by  $C(x)\in \{1,2,3\}$ the true (unknown) label, and $\hat{C}(x)\in \{1,2,3\}$ the label assigned by the algorithm to observation $x$. Then the classification error is given by

\begin{equation}\label{error} 
\min_{\pi} \frac{1}{n}\sum_{i=1}^{n}\mathbf{1}_{\Big\{C (x_i)\neq  \pi(\hat{C} (x_i))\Big\}}.
\end{equation}

Figure \ref{fig_mediasc} shows the performance of RobustkM (with $K=10$), $k$-means, PAM, and TClust (with $\alpha=0.01$), over $1000$ replications.  In TClust, the  trimmed data (at the end of the algorithm) are assigned to the nearest centres.  This toy example shows that the proposed algorithm is a competitive alternative to other methods that ``robustify''  $k$-means.

\begin{figure}[htb]
\centering
 \subfloat{\includegraphics[width=120mm]{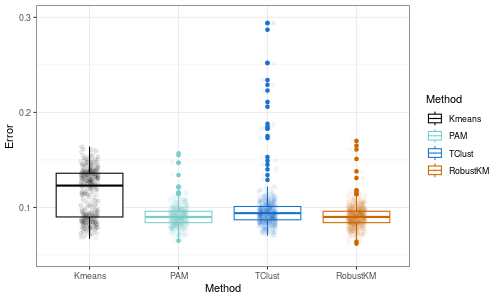}}
\caption{Box plot of classification errors, according to \eqref{error},  of $K$-means, TClust, PAM and RobustKM over  $1000$ replicates.}
\label{fig_mediasc}
\end{figure}

\subsection{Bandits}
The bandit problem was first proposed by Thompson in 1933, and has been recently been gaining increasing  attention. There have been several adaptations to various frameworks, as evidenced by comprehensive surveys, such as \cite{boursier2022survey} and \cite{burtini2015survey}, which cover a broad spectrum of applications.

According to \cite{lattimore2020bandit}, a bandit problem can be described as a sequential game played between a learner and an environment, spanning over a  number of rounds, $T$, which is a positive integer. In each round, denoted by $t$, the learner  first chooses an action (denoted by $\mathcal{A}_t$), from a defined set of actions (denoted by $\mathcal{A}$). Following this, the environment unveils a reward $X_{t}$, which is a real number. Given the action $\mathcal{A}_t$, the reward $X_t$ is assumed to be independent of the past.

In the stochastic $L$-armed bandit problem, there are $L$ possible actions. The rewards are derived from a set of distributions denoted by $P_1$ through $P_L$ (not depending on $t$).
The aim of the learner is to identify an arm, denoted by  $i^*$, whose distribution yields the highest mean, denoted by $\mathbb{E}(P_{i^*})$. This is equivalent to minimizing the regret
$$R_T=T\mathbb{E}(P_{i^*})-\sum_{t=1}^T \mathbb{E}(P_{\mathcal{A}_t}).$$ 

A classic example of such an algorithm is the well-known ``upper confidence bands'' (UCB). Let us recall that, in the UCB, the first choice $\mathcal{A}_1$ is made at random, and for $t\in \{1,\dots,T-1\}$,
\begin{equation}\label{UCB0}
\mathcal{A}_{t+1}=\underset{j\in \{1,\dots,L\}}{\text{argmax}} \ \ \frac{X_{1}\mathbb{I}_{\{\mathcal{A}_1=j\}}+\dots+ X_{t}\mathbb{I}_{\{\mathcal{A}_t=j\}}}{N_{t,j}}+\sqrt{\frac{\log(t)}{N_{t,j}}}
\end{equation}
where $X_{i}\mathbb{I}_{\{\mathcal{A}_1=j\}}$ is the reward of arm $j$ at time $i$ if this arm is chosen  at that time and $N_{t,j}=\sum_{s=1}^t \mathbb{I}_{\{\mathcal{A}_s=j\}}$ is the number of times that the arm $j$ was chosen.

\subsubsection{Heavy-tailed bandits}
To get a logarithmic regret for the UCB (i.e., $R_T/\log(T)\to C$), it is usually assumed that the distributions $P_i$ are sub-Gaussian. This can be weakened to require the distributions to have finite moment generating function, see \cite{agrawal1995sample}. To overcome this limitation, instead of just estimating the mean at each step $t$ of the UCB, in \cite{bubeck2013bandits} it is proposed to use a robust estimator, $\hat{\mu}_{i,t}$, of $\mathbb{E}(P_i)$, that fulfills the following assumption:\vspace{.5cm}

\noindent \textbf{Assumption 1} Let $\epsilon \in(0,1]$ be a positive parameter, and let $c,v$ be a positive constant. Let $X_1,\dots,X_n$ be i.i.d. random variables with finite mean $\mu$. Suppose that for all $\delta\in (0,1)$ there exists an estimator $\hat{\mu}=\hat{\mu}(n,\delta)$ such that, with probability at least $1-\delta$,
\begin{equation}\label{cond1}
\hat{\mu}\leq  \mu+v^{1/(1+\epsilon)}\left(\frac{c\log(1/\delta)}{n}\right)^{\epsilon/(1+\epsilon)}
\end{equation}
and also, with probability at least $1-\delta$,
\begin{equation}\label{cond3}
\mu\leq  \hat{\mu}+v^{1/(1+\epsilon)}\left(\frac{c\log(1/\delta)}{n}\right)^{\epsilon/(1+\epsilon)}.
\end{equation}
In that case we say that $\hat{\mu}$ fulfills Assumption 1.

\begin{remark}\label{bandits} From  \eqref{concentration}, our robust proposal $\mu^*$, applied on each arm, and based on $K=\lceil 8\log(1/\delta)\rceil$ groups, fulfills Assumption 1, for $\epsilon=1$, $v=\mathbb{V}(P_i)$, ($\mathbb{V}(P_i)$ being the variance of the distribution $P_i$), and $c=16$. In that case, for an arm $i=1,\dots,L$, $n=N_{t,i}/K$.
\end{remark}

\cite{bubeck2013bandits} proposes the following robust variant of the UCB: given $\epsilon \in (0,1]$, for arm $i$, define $\hat{\mu}_{i,s,t}$ as the estimator $\hat{\mu}(s,t^{-2})$ based on the first $s$ observed values $X_{i,1},\dots,X_{i,s}$ of the reward of arm $i$. Define the index
$$B_{i,s,t}=\hat{\mu}_{i,s,t}+v^{1/(1+\epsilon)}\left(\frac{c\log(t^2)}{s}\right)^{\epsilon/(1+\epsilon)}$$
for $s,t\geq 1$ and $B_{i,0,t}=+\infty$. Then, at time $t$ draw an arm maximizing $B_{i,N_{i,t-1},t}$.
\vspace{.5cm}

We propose the following algorithm. First we choose the arms at random from $t=1,\dots,t_0$, where  $t_0$ guarantees that for all $i=1,\dots, L$, $N_{t_0,i}/\lceil 8\log(t_0^2)\rceil \geq 1$. We compute, for each arm, the estimator \eqref{eq:def}, denoted by $\mu^*_{t_0,i}$, where we split the $N_{t_0,i}$ observations into $K=\lceil 8\log(t_0^2)\rceil$ groups, and compute the mean of each group. Define the index
$$\mathcal{B}_{i,N_{t_0,i},t_0}=\mu^*_{t_0,i} +4\sqrt{\hat{\mathbb{V}}(P_i)} \left(\frac{\log(t_0^2)}{N_{t_0,i}}\right)^{1/2},$$
where $\hat{\mathbb{V}}(P_i)$ is any consistent estimation of $\mathbb{V}(P_i)$, or an upper bound. At time $t_0+1$, we choose the arm that maximize $\mathcal{B}_{i,N_{t_0,i},t_0}$, at time $t_0+2$ we choose the arm that maximize $\mathcal{B}_{i,N_{t_0+1,i},t_0+1}$, and so on.

Proposition 1 in \cite{bubeck2013bandits} proves that   this algorithm attains logaritmic regret. More precisely,
$$R_T\leq \sum_{i:\Delta_i>0} \left( 32\left(\frac{\mathbb{V}(P_i)}{\Delta_i}\right) \log(T)+5\Delta_i\right),$$
where $\Delta_i=\mathbb{E}(P_{i^*})-\mathbb{E}(P_i)$.

\subsubsection{Simulations}

 As a toy example, let us consider the classical two-armed bandit problem  and rewards given by $X_t \vert \mathcal{A}_t =j \sim \mu_j + S(3),$ for $j=1,2$, where $S(3)$ is a random variable following  Student's distribution with $3$ degrees of freedom, $\mu_1=7$, and $\mu_2=8$. In our algorithm, indicated by RUCB in Figure \ref{fig_bandits},  the first $t_0=40$ arms are chosen at random.

The results are shown in Figure \ref{fig_bandits}, where the red dotted horizontal line $(y=8)$ is the maximum expected gain. The dashed lines corresponds to the mean rewards of the UCB (orange) and the RUCB (blue) respectively.  As can be seen, it takes the RUCB algorithm about 120 steps to outperform the UCB algorithm, and the difference grows larger as the number of steps increases.

\begin{figure}[htb]
\centering
 \subfloat{\includegraphics[width=140mm]{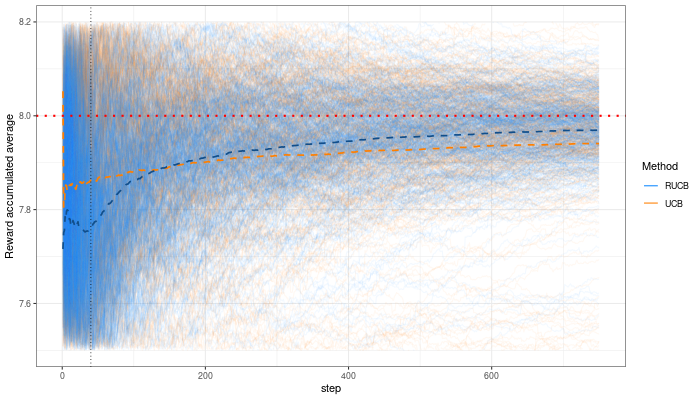}}
\caption{Cumulative gains over 500 replications, for $t=1, \ldots,750$. The red dotted horizontal line ($y=8$) is the maximum expected gain. The black dotted vertical line ($x=40$) indicates the number of random warm-up runs in the RUCB  algorithm. The dashed lines depict the mean reward of the UCB (orange) and RUCB (blue) algorithms.}
\label{fig_bandits}
\end{figure}

\subsection{Robust regression}
A current topic of interest in statistics is that of regression models in the presence of noise. It is known that a small fraction of outliers can cause severe biases in classical regression estimators.   These classical models generally assume additive residuals in the model with finite second moment (e.g., Gaussian).   An extensive review of robust outlier estimation methods for nonparametric regression models is provided in \cite{salibian2023}.

We tackle the problem of estimating   the function $m : \mathcal{X}\to \R$, from an i.i.d. sample  of a random element $(X,Y)\in \mathcal{X}\times \mathbb{R}$, that satisfies the model
 
\[
Y=m(X)+\epsilon
\]
where $\epsilon$ is a noise term such that $\esp{\epsilon|X}=0$. To keep the following discussion fairly simple, we will only cover the case $\mathcal{X}=\R$. In the context where very little information is known about $m$, a general estimator is often given by the kernel estimator
 \begin{equation}
 \label{eq:def_kernel_estim}
      \hat{m}(x)=\frac{\sum_{i=1}^n K_h(X_i-x)Y_i}{\sum_{i=1}^n K_h(X_i-x)},
 \end{equation}

 \noindent where $K_h (x) = h^{-1}\phi(h^{-1}x)$ and $h>0$ is some bandwidth parameter. The function $\phi$ is non-negative and such that $\int_{\R}\phi(x)dx=1$. In general, the continuity of the function $m$ is enough to get the  
consistency of the kernel estimator $\hat{m}$ and a light tailed behavior for $\epsilon$ gives sub-exponential deviation bounds for the estimator around its mean value $m$. Nevertheless, in various concrete settings, one can face heavy tailed distributions for $\epsilon$ in such a way that the estimator proposed in  \eqref{eq:def_kernel_estim} becomes highly unstable. Indeed, the presence of (virtually) one outlier is enough to drive the estimator towards values very distant from $m$.
One natural way to measure the quality of the estimator $\hat{m}$ is to consider the $L_2$ distance 
\[
d_2(\hat{m},m)=\esp{|\hat{m}(X)-m(X)|^2}^{1/2}.
\]
In this context, it is possible to introduce the robust version of the estimator by considering GROS  with the distance $d_2$.
In practice, the distance $d_2$ is actually intractable. To overcome this difficulty, it is common to use a discretization on a mesh for the space $\mathcal{X}$ and approximate the integral by a Riemann sum.
This estimator verifies a bound as in Theorem \ref{thm:aggregation_theoretical} for the associated distance.
The estimator $m^*$ is, in the sense of Definition \eqref{eq:def}, the best candidate in the class of the estimator $m_1,\dots,m_K$ constructed on the disjoint groups $G_1,\dots,G_K$ of data points. These estimators $m_1,\dots,m_K$ are defined by 
\[
m_j(x)=\frac{\sum_{i\in G_j} K_h(X_i-x)Y_i}{\sum_{i\in G_j} K_h(X_i-x)},
\]
so that these estimators are independent.
Following subsection \ref{missesp}, the following estimator for $\hat{m}$ can be proposed:
\[
\hat{m}=\argmin_{m\in \tilde{\mathcal{M}}} \min_{I:|I|> \frac{K}{2}} \max_{j\in I} d_2(m_j,m),
\]
where the $(m_j)_{j=1}^K$ are the kernel estimates of $m$ on each of the $K$ groups, and the set $\tilde{\mathcal{M}}$ denotes the set of functions that are piece-wise equal to one of the $m_i$.
The difference from the naive definition is that the set of minimization does not end with an estimator of the form $m_{i^*}$, which allows avoiding choosing a function that may be a good fit in some regions of the space but sensitive to outliers in other parts.

In order to fully use the result of Equation \eqref{concentration2}, we cite a result that gives an upper bound for the mean squared error for any of the estimators $m_j$ previously defined.
In chapter 5 of \cite[Theorem 5.2]{gyorfi2002distribution}, we get that if $m$ is $\lambda$-Lipschitz, and that $\text{Var}(Y|X=x)\le \sigma^2$ for all $x\in \mathcal{X}$, then
\[
\esp{\|m_1-m\|^2}\le c\left(\frac{\sigma^2+\sup_{x\in \mathcal{X}}m(x)^2}{(n/K)h^d}\right)+\lambda^2h^2.
\]
This induces a choice of $h$ of the form
\[
h=c'\left(\frac{K}{n}\right)^{\frac{1}{d+2}}.
\]
The upper bound then takes the form of
\[
\esp{\|m_1-m\|^2}\le c''\left(\sigma^2+\sup_{x\in \mathcal{X}}m(x)^2\right)^{2/d+2}\times \left(\frac{n}{K}\right)^{-2/d+2},
\]
which can be directly plugged into the main bound of the theorem.

\subsubsection{Simulations}

We compare, by means of simulations, the performance of GROS  with  some classical and robust regression alternatives proposed in the literature.

In this example we consider a sample $X_1, \ldots, X_{1000}$ with uniform distribution on $[0,5]$. Let  $m(x)= 4\sin(3x)$ and suppose that the centered noise follow the  skew-normal Student distribution,  see \cite{fernandez1998, azzalini2013}, whose density is defined as follows: denote by $t(x,\kappa,\nu)$ the density of the non-standardized Student's distribution with  location $\kappa$ and and $\nu$ degrees of freedom, and with cumulative distribution function $T(x,\kappa,\nu)$. Then,  the density of the skew-normal Student is $$ f(x;\kappa, \nu,\sigma, \xi )= \frac{1}{\sigma}t(x/\sigma,\kappa,\nu)T \left(\frac{\xi x}{\sigma},\kappa, \nu \right),$$

\noindent where $\sigma  > 0$ denotes the  scale parameter. The slant (or skewness parameter) is $\xi$.

We will write NW for the non-parametric Nadaraya--Watson kernel regression estimator \eqref{eq:def_kernel_estim}, see \cite{nadaraya1964,watson1964}.  As robust estimators, we consider the proposals developed in  \cite{oh2007}  (ONL estimator)  and \cite{boente2017} (SBMB estimator). Both estimators are implemented in the \textbf{R} language: the first  in the \textbf{fields} library and the other in the \textbf{RBF} library. Our estimator in this context will be called RANW (Robust Aggregation for Naradaya--Watson). For the latter, $K=12$ was considered. Figure \ref{fig_cajas1} shows the estimated regressions in each scenario and in red the true  function $f$. It is clear that the NW estimator performs poorly in all cases.

The performance of each estimator is measured by the average distance $d_2(\hat{m},m)$ over $1000$ replicates, considering $\kappa=0$, $\nu=3$ (note that the noises have heavy tails) and varying the parameters $\sigma$ and $\xi$. For the bandwidth parameter $h$, we choose $0.2$ in all cases.

Figure \ref{fig_cajas1} shows box plots of the errors for four choices of the parameters. As can be seen,  when the noise is asymmetric (i.e., $\xi\neq 1$), the best-performing estimators are RANW and ONL,  depending on the value of the scale parameter.  Note that when the distribution of the noise is symmetric (i.e., $\xi=1$),   SBMB and  RANW perform the best.

\begin{figure}[htb]
\centering
 \subfloat[$\sigma=9, \nu=3, \xi=1$]{\includegraphics[width=72mm]{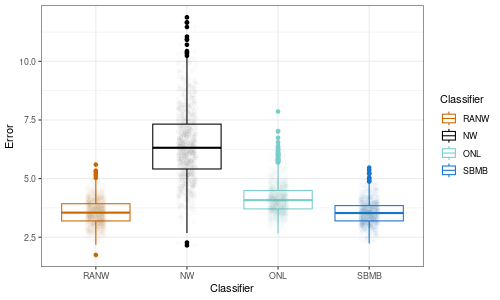}}
  \subfloat[$\sigma=9, \nu=3, \xi=9$]{\includegraphics[width=72mm]{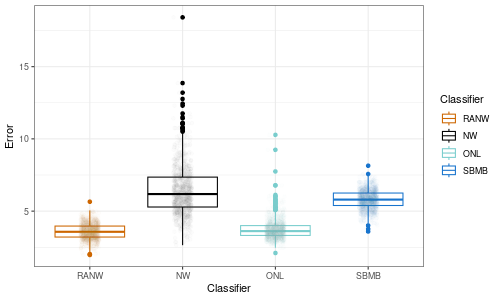}}
  \\
   \subfloat[$\sigma=16, \nu=3, \xi=1$]{\includegraphics[width=72mm]{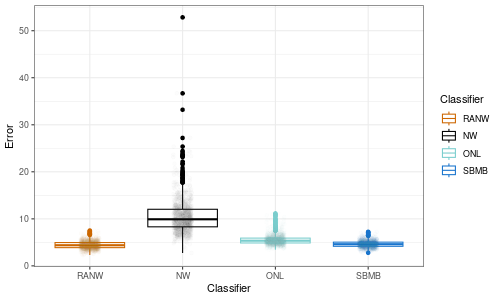}}
    \subfloat[$\sigma=16, \nu=3, \xi=9$]{\includegraphics[width=72mm]{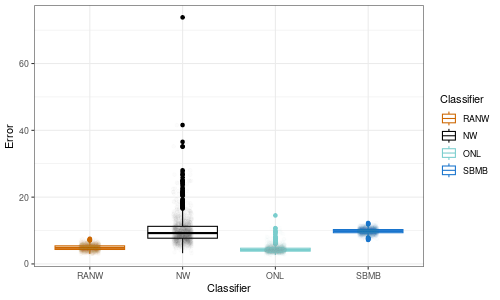}}
\caption{Box plot of classification errors (according to L2 distance) in $1000$ replicates.  The different scenarios are obtained in the skew-normal Student distribution with $\sigma \in \{ 9,16\}$ and $\xi \in \{ 1,9\}$, fixed $\nu=3$ and $\kappa=0$.}
\label{fig_cajas1}
\end{figure}

\begin{figure}[htb]
\centering
 \subfloat[$\sigma=9, \nu=3, \xi=1$]{\includegraphics[width=70mm]{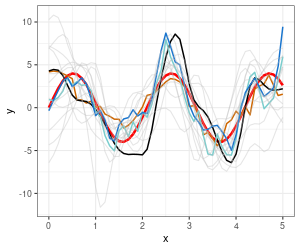}}
  \subfloat[$\sigma=9, \nu=3, \xi=9$]{\includegraphics[width=70mm]{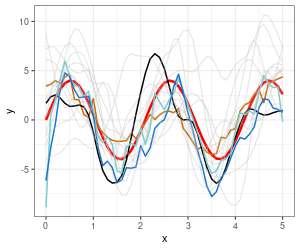}}
  \\
   \subfloat[$\sigma=16, \nu=3, \xi=1$]{\includegraphics[width=70mm]{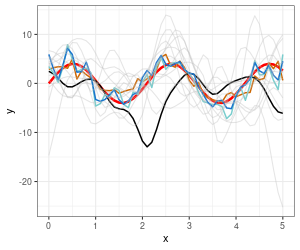}}
    \subfloat[$\sigma=16, \nu=3, \xi=9$]{\includegraphics[width=70mm]{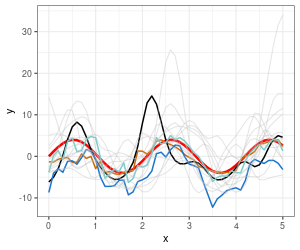}}
\caption{Regression functions estimated with the RANW (orange), NW (black), ONL (light blue) and SBMB (blue) in one replicate. The true function is shown in red.  The different scenarios are obtained in the skew-normal Student distribution with $\sigma \in \{ 9,16\}$ and $\xi \in \{ 1,9\}$, fixed $\mu=0$ and $\nu=3$. }
\label{fig_cajas2}
\end{figure}

\subsection{Robust set estimation}\label{sec:setesti}

Set estimation involves determining a set, or a function of that set, based on a random sample of points. This set could represent various things, such as the support of a probability distribution (see, for instance, \cite{rod07,chola:14}), its boundary (see, for instance \cite{cue:04}, the surface area of the boundary (see, for instance \cite{aaron22}), to name a few.

Within this framework, shape constraints are usually imposed, convexity being the most well-known. However, it can be restrictive for some applications, such as, for instance, if the set is the home-range of a species  (see \cite{cho16}, \cite{ch21}, \cite{manu} and references therein). Usually, the available data is an i.i.d. sample of a random vector whose support is the unknown set. For classic estimators such as the convex hull, $r$-convex hull, or cone-convex hull, any noise in the sample, no matter how small, drastically changes the estimators. This becomes even more pronounced in the case of convexity.

\subsubsection{Simulations}

We show the performance of our robust proposal under a noisy model, where the aim is to estimate a convex set.  More precisely, let $D(r,R)$  be the uniform distribution on the ring in $\mathbb{R}^2$ with inner radius $r$ and outer radius $R$. We simulated $2000$ i.i.d. observations of the mixture

$$ (1-\lambda)D(0,1) + \lambda D(1, 1.25).$$

The aim is to estimate $D(0,1)$ from this sample.  We have chosen $\lambda=0.01$ as the proportion of noise. The estimator, referred to as RChull, is the one proposed in  Section \ref{sec:compasp}, and we considered the Hausdorff distance between compact sets to measure the discrepancy between $D(0,1)$ and the estimator. To build our estimator, we first split the original sample at random into $K=20$ disjoint groups of points of size $100$, and compute the convex hull on each group. Then, we select from the $K$ hulls $H_1,\dots,H_{20}$, the hull $H_{j^*}$, with $j^*$ as in \eqref{eq:def_jstar}.

To gain an insight into the improvement resulting from using this  robust procedure, we show on the left of Figure \ref{fig_set} the set to be estimated, the classical estimator (the convex hull (Chull) of the whole sample) and the 20 hulls of the subsamples. The right panel shows the convex hull of the whole sample, together with $H_{j^*}$.

\begin{figure}[htb]
\centering
  \subfloat{\includegraphics[width=60mm]{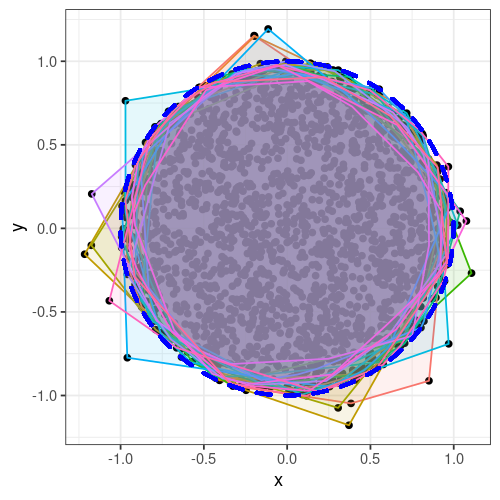}}
 \subfloat{\includegraphics[width=72mm]{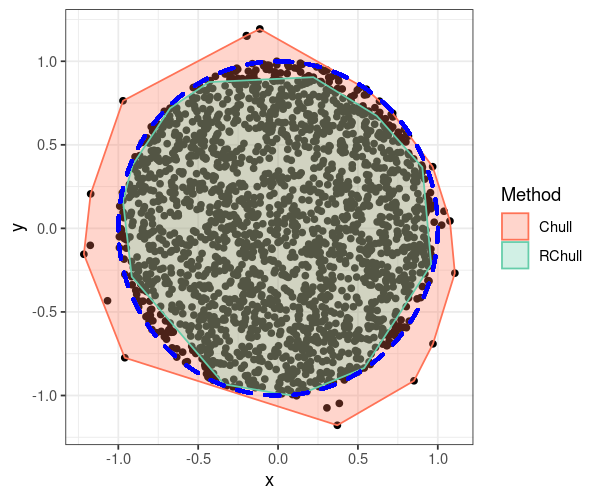}}
\caption{In the blue dotted line the boundary of the ball of radius 1. The sample is shown as solid black points. Outside this ball the sample generated from $D(1,1.25)$. On the left there are shown the 20 convex hulls of the selected subsamples (of size 100). On the right the convex hullf of the whole sample (Chull) and the robust estimator based on the $20$-convex hulls, (RChull).}
\label{fig_set}
\end{figure}

To evaluate the performance of GROS , we run 100 replicates and estimate the set by both methods (the classical convex hull  and our robust proposal) in each replicate.  We calculate the Hausdorff distance of the estimated sets RChull and Chull from the set D(0,1). In Figure \ref{box_set}  box plots of these distances are shown. The RChull estimator outperforms the classical Chull estimator, but as can be seen, it has  larger variability.

\begin{figure}[htb]
\centering
  \subfloat{\includegraphics[width=80mm]{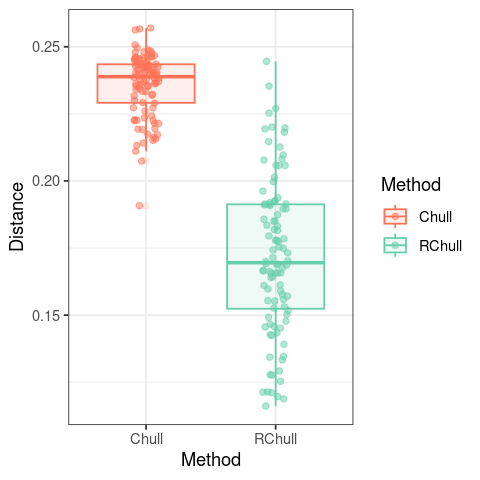}}
\caption{}
\label{box_set}
\end{figure}

\subsection{Robust persistent diagram}\label{sec:peresti}

The robustness of persistent homology to perturbations in data measured by the Hausdorff distance is well established. However, it has a high sensitivity to outliers, as discussed in \cite{vishwanath2020, vishwanath2022}.

In this section, we introduce a robust persistence diagram, which we call the \textit{Robust Wasserstein Estimator}, using \eqref{eq:def_jstar}.

The measure of dissimilarity between two persistence diagrams $P_1$ and $P_2$  is quantified by the 1-Wasserstein distance $W_1(P_1,P_2)$. This distance quantifies the cost associated with achieving the optimal alignment of points between the two diagrams, as detailed in \cite[p. 202]{edelsbrunner2022}.

\subsubsection{Simulations}
The example examines uniformly simulated data on $S^1$ consisting of 600 points (baseline sample). It explores two potential scenarios of sample distortion, as depicted in Figure \ref{xx}.

\begin{figure}[htb]
\centering
  \subfloat[Baseline sample]{\includegraphics[width=46mm]{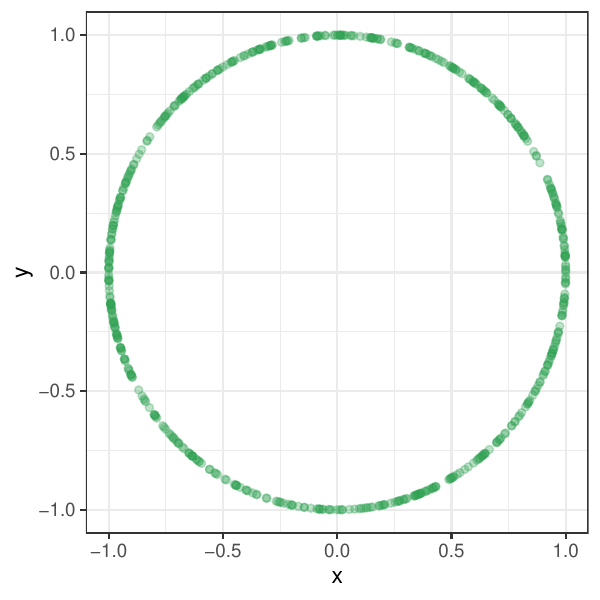}}
 \subfloat[Scenario 1: Local perturbation]{\includegraphics[width=47mm]{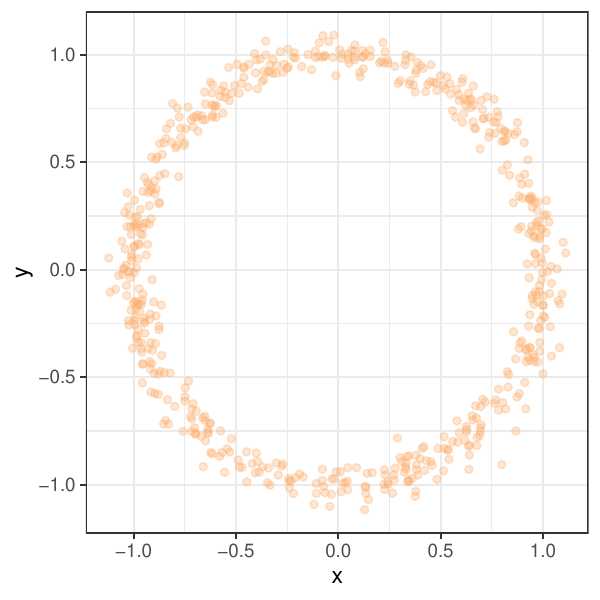}}
 \subfloat[Scenario 2: Groups of outliers]{\includegraphics[width=47mm]{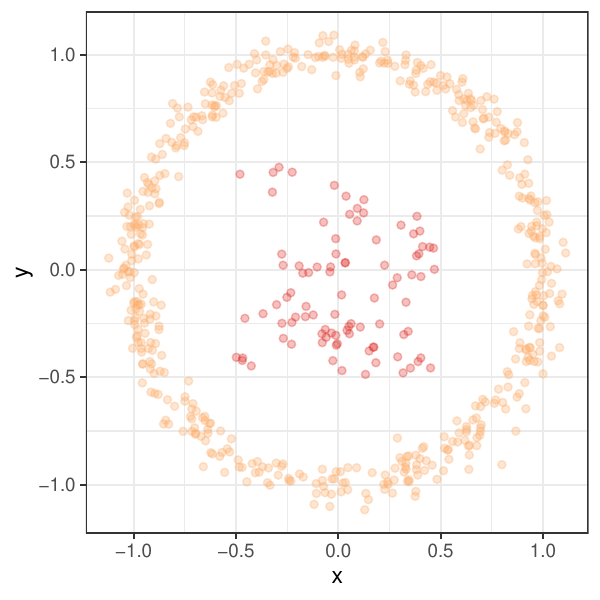}}
\caption{
(a) Baseline sample of 600 points uniformly distributed on $S^1$. (b) Locally perturbed sample as described in Scenario 1. (c) Sample perturbed in accordance with Scenario 2.}
\label{xx}
\end{figure}

\begin{itemize}
    \item \textbf{Scenario 1: } Local perturbation. The original sample is perturbed using Gaussian noise centered at each sample point, with a standard deviation matrix $0.05\times Id$.
    \item \textbf{Scenario 2:}  Group of Outliers: We randomly selected $90\%$ of the perturbed sample as defined in Scenario 1. The remaining $10\%$ are derived from a Matérn cluster process within the square region $[-0.5,0.5]^2$. This process is characterized by an intensity of $3$ for the Poisson process of cluster centers, a scale of $0.25$, and an average of $20$ points per cluster.

\end{itemize}

The persistence diagrams for the previously described tree samples were computed. Figure \ref{rr} displays these diagrams, where $Dgm$, $Dgm_1$, and $Dgm_2$ represent the persistence diagrams of the baseline sample, the locally perturbed sample (Scenario 1), and the sample with outliers (Scenario 2), respectively. It is evident that the persistence diagram is more distorted in scenarios with groups of outliers than it is  with those with local perturbations.

\begin{figure}[htb]
\centering
  \subfloat[$Dgm$]{\includegraphics[width=47mm]{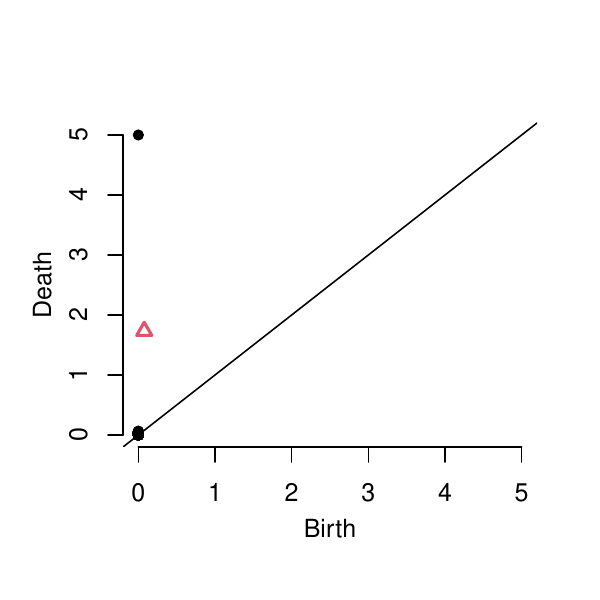}}
 \subfloat[$Dgm_1$]{\includegraphics[width=47mm]{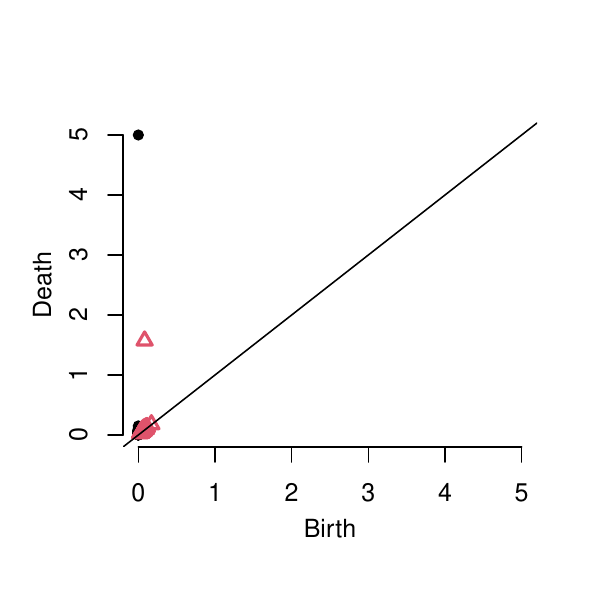}}
 \subfloat[$Dgm_2$]{\includegraphics[width=47mm]{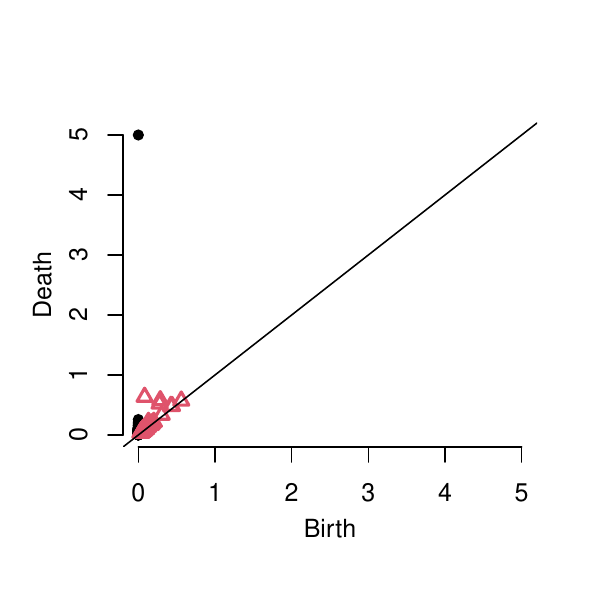}}
\caption{(a): Persistence diagram of the baseline sample. (b) and (c): Persistence diagrams for the samples perturbed according to Scenarios 1 and 2, respectively.}
\label{rr}
\end{figure}

To evaluate the performance of our robust proposal \eqref{eq:def_jstar}, the sample is divided into \( K=6 \) distinct groups. The robust persistence diagrams for Scenarios 1 and 2 are labeled as \( RDgm_{1} \) and \( RDgm_{2} \), respectively. We computed $W_1(Dgm_1, Dgm)$, $W_1(Dgm_2, Dgm)$, $W_1(RDgm_1, Dgm)$, and  $W_1(RDgm_2, Dgm)$.

A total of 100 independent iterations of this experiment were conducted, and box plots of the respective distances are shown in Figure \ref{box_diag}. The results indicate that in both scenarios, the robust estimates of the persistence diagrams show improved performance, being closer to the baseline diagram in terms of the Wasserstein distance.

\begin{figure}[htb]
\centering
  \subfloat{\includegraphics[width=130mm]{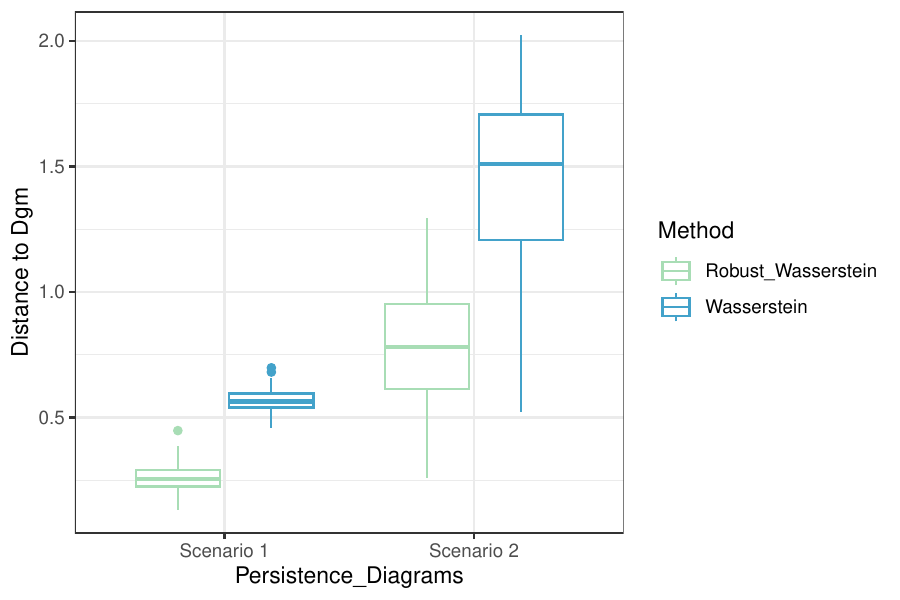}}
\caption{Box plots illustrating the distances between the persistence diagrams of perturbed samples and the baseline sample diagram \(Dgm\), as well as the distances between the robust persistence diagrams and the baseline sample diagram.}
\label{box_diag}
\end{figure}

\section{Concluding remarks}

We have demonstrated through simulations that GROS , which is applicable to a broad range of problems, significantly improves upon the non-robust, problem-specific solutions for each of the five examples treated. It is also competitive with robust solutions designed for each specific case, even showing some improvements. GROS  proposal's flexibility makes it applicable to a wide variety of problems, including those already presented, as well as any other scenario where robustness plays a crucial role. Furthermore, more examples using only pseudo-distances may be of interest for future research.

%
%

\end{document}